\documentclass[12pt]{article}
\usepackage{mathrsfs}
\usepackage{amsmath, amssymb, amsthm, latexsym, amsfonts, epsfig, color,graphicx}

\begin{document}

\newcommand{\s}{\sigma}
\renewcommand{\k}{\kappa}
\newcommand{\p}{\partial}
\newcommand{\D}{\Delta}
\newcommand{\om}{\omega}
\newcommand{\Om}{\Omega}
\renewcommand{\phi}{\varphi}
\newcommand{\e}{\epsilon}
\renewcommand{\a}{\alpha}
\renewcommand{\b}{\beta}
\newcommand{\N}{{\mathbb N}}
\newcommand{\R}{{\mathbb R}}
\newcommand{\eps}{\varepsilon}
\newcommand{\EX}{{\Bbb{E}}}
\newcommand{\PX}{{\Bbb{P}}}

\newcommand{\cF}{{\cal F}}
\newcommand{\cG}{{\cal G}}
\newcommand{\cD}{{\cal D}}
\newcommand{\cO}{{\cal O}}

\newcommand{\de}{\delta}

\newcommand{\grad}{\nabla}
\newcommand{\n}{\nabla}
\newcommand{\curl}{\nabla \times}
\newcommand{\dive}{\nabla \cdot}

\newcommand{\ddt}{\frac{d}{dt}}
\newcommand{\la}{{\lambda}}

\newtheorem{theorem}{Theorem}[section]
\newtheorem{lemma}{Lemma}[section]
\newtheorem{remark}{Remark}[section]
\newtheorem{example}{Example}[section]
\newtheorem{definition}{Definition}[section]
\newtheorem{corollary}{Corollary}[section]
\newtheorem{assumption}{Assumption}[section]
\newtheorem{prop}{Proposition}[section]
\newtheorem{notation}{Notation}
\makeatletter
\@addtoreset{equation}{section}
\makeatother
\renewcommand{\theequation}{\arabic{section}.\arabic{equation}}

\makeatletter\def\theequation{\arabic{section}.\arabic{equation}}\makeatother

\title{A Wong-Zakai Approximation of Stochastic Differential Equations Driven by a General Semimartingale}

\author{
{ \bf\large Xianming Liu}\\
{\it\small School of Mathematics and Statistics},\\
{\it\small Huazhong University of Sciences and Technology},
{\it\small Wuhan 430074,  China}\\
{\it\small E-mail:    xmliu@hust.edu.cn}\vspace{1mm}\\ \\
{ \bf\large Guangyue Han}
\vspace{1mm}\vspace{1mm}\\
{\it\small  Department of Mathematics},\\
{\it\small The University of Hong Kong},
{\it\small Hong Kong, China}\\
{\it\small E-mail:    ghan@hku.hk}\vspace{1mm}}

\date{\today }
\maketitle

\begin{abstract}
We examine a Wong-Zakai type approximation of a family of stochastic differential equations driven by a general c\`{a}dl\`{a}g
semimartingale. For such an approximation, compared with the pointwise convergence result by Kurtz, Pardoux and Protter~\cite[Theorem 6.5]{Kurtz}, we establish stronger convergence results under the Skorokhod $M_1$-topology, which, among other possible applications, implies the convergence of the first passage time of the solution to the stochastic differential equation.
\end{abstract}

\noindent {\it \small Key words}: {\small Wong-Zakai approximation, stochastic differential equation, semimartingale,
the Skorokhod $M_1$-topology, random time change}

\setcounter{secnumdepth}{5} \setcounter{tocdepth}{5}

\makeatletter
    \newcommand\figcaption{\def\@captype{figure}\caption}
    \newcommand\tabcaption{\def\@captype{table}\caption}
\makeatother

\section{Introduction}  \label{S1}
Let $L=\{L(t); 0\leq t < \infty\}$ be a stochastically continuous
c\`{a}dl\`{a}g semimartingale~\cite{Protter} defined on a
probability space $(\Omega, \mathcal{F}, \mathcal{F}_t,
\mathbb{P})$. For any $\epsilon > 0$, let $L^\epsilon$ be the smooth
approximation~\cite{Ikeda,Kurtz,Zhang13} of $L$ defined by
\begin{equation} \label{apps1}
L^\epsilon(t) :=  \frac{1}{\epsilon}\int_{t-\epsilon}^t L(s)
ds, \quad 0 \leq t < \infty,
\end{equation}
and let $X^\epsilon=\{X^\epsilon(t); 0\leq t < \infty\}$ be the solution to the following random differential equation
\begin{equation} \label{mainrde}
dX^{\epsilon} (t) = b(X^{\epsilon}(t))dt +f(X^{\epsilon} (t)) d
L^\epsilon(t),  \quad X^{\epsilon} (0)=X_0,
\end{equation}
where $b(\cdot), f(\cdot)$ are some functions from $ \mathbb{R}$ to $\mathbb{R}$ satisfying certain regularity conditions, and $X_0$ is an $\mathcal{F}_0$-measurable random variable.

In this paper, we are concerned with the convergence behavior of $X^{\epsilon}$ as $\epsilon$ tends to $0$. Since the equation~\eqref{mainrde} is a perturbed version of \eqref{mainsde} in the sense of Wong-Zakai~\cite{Hairer,Kurtz,Marcus,Tes,WZ65a,WZ65b}, one naturally speculates that $X^{\epsilon}$ converges to $X$ in some sense, where $X$ is the solution to the following stochastic differential equation:
\begin{equation}\label{mainsde}
X(t) = X_0+ \int_0^t b(X(s))ds + \int_0^tf(X(s-))\diamond  d L(s),
\end{equation}
where $\diamond$ denotes Marcus integral. Note that the equation (\ref{mainsde}) is in fact a Marcus canonical equation and can be equivalently rewritten as
\begin{eqnarray}\label{m1}
X(t)&=X_0+\int_0^t b(X(s))ds+ \int_0^tf(X(s))\circ dL^c(s)+ \int_0^tf(X(s-)) dL^d(s) \nonumber  \\
    &\quad +\sum_{0< s\leq t}[\phi(\Delta L(s)f;X(s-), 1)-X(s-)-f(X(s-))\Delta L(s)],
\end{eqnarray}
where $L^c$ and $L^d$ are respectively the continuous and discontinuous
parts of $L$, and $\circ$ denotes Stratonovich
differential, and furthermore $\phi(\sigma; u,t)$ is the flow generated by a
vector field $\sigma$:
\begin{equation}\label{m2}
\frac{d\phi(\sigma; u,t)}{dt}=f[\phi(\sigma; u,t)], \quad \phi(\sigma; u,0)=u.
\end{equation}
For more details about Marcus integral and canonical equations, we refer the reader to~\cite{Applebaum,Kunita 99,Kunita
96,Kunita2004,Kurtz}.

For the special case $b=0$, it has been shown by Kurtz, Pardoux and Protter~\cite{Kurtz} that for all but countably many $t$, $X^\epsilon(t)$ converges in probability to $X(t)$, as $\epsilon$ tends to $0$.  As detailed in the following theorem, we will show that for any $T > 0$, $\{X^{\epsilon}(t); 0 \leq t \leq T\}$ converges in probability to $\{X(t); 0 \leq t \leq T\}$ in $D([0, T], \R)$~\footnote{Following the usual practice in the theory of stochastic calculus, we will assume that the sample paths of any stochastic process in this paper are all c\'{a}dl\`{a}g.}, the space of all the c\`{a}dl\`{a}g functions over $[0, T]$, under the the Skorokhod $M_1$-topology, or simply put, $X^{\epsilon}$ converges in probability to $X$ under the Skorokhod $M_1$-topology. Here we remark that the same convergence is not possible under the Skorokhod $J_1$-topology due to the simple fact that convergence under the $J_1$-topology keeps the
continuity, whereas $X^{\epsilon}$ is continuous for any $\epsilon > 0$ and $X$ may be discontinuous. For the precise definitions of the Skorokhod $J_1$ and $M_1$-topologies, see Appendix~\ref{skorokhod}.
\begin{theorem}\label{maintheorem}
Suppose that the functions $b(\cdot)$, $f(\cdot)$ and $f'(\cdot)$ are bounded and Lipschitz. Then, as $\epsilon$ tends to $0$, $X^\epsilon$ converges in probability to $X$ under the Skorokhod $M_1$-topology.
\end{theorem}

For a real-valued stochastic process $Y=\{Y(t); 0\leq t < \infty\}$ and a positive real number $a>0$, let $\tau_a(Y)$ denote the first passage time of $Y$ with respect to $a$, that is,
$$
\tau_a(Y) = \inf\{t\geq0: Y(t)>a\}.
$$
As an immediate corollary of Theorem~\ref{maintheorem}, we have
\begin{corollary}
For any positive real number $a > 0$, $\tau_a(X^\epsilon)$ converges
to $\tau_a(X)$ in distribution.
\end{corollary}

\begin{proof}
The corollary follows from Theorem \ref{maintheorem}, the easily verifiable fact that for any $t, \epsilon > 0$,
$$
\PX(\tau_a(X^\epsilon) \leq t)=\PX(\sup_{s\in [0,t]} X^\epsilon(s)
\geq a), \quad \PX(\tau_a(X) \leq t)=\PX(\sup_{s\in [0,t]} X(s)
\geq a),
$$
and the fact that the supremum functional as above is continuous under the Skorokhod
$M_1$-topology~\cite[Lemma 2.1]{Puhalskii}.
\end{proof}

The key tool that we used in this work is the so-called method of
random time change (see, e.g.,~\cite{Kurtz91}), which is a
well-known method that has also been used in~\cite{Kurtz}. On the
other hand though, the power of this method somehow has not been
fully utilized in~\cite{Kurtz}: Theorem~\ref{maintheorem} in this
work, which is established through a short and simple
argument, immediately implies that $X^\epsilon(t)$ converges in
probability to $X(t)$ for all but countably many $t$, which further
implies Theorem $6.5$ in \cite{Kurtz}. As a matter of fact, the
power of this method can be further showcased in some special
setting: For the case that $L$ is a L\'{e}vy process, the
method of Hintze and Pavlyukevich~\cite{Ilya} can be adapted to show
that as $\epsilon$ tends to $0$, $L^\epsilon$ converges in
probability to $L$ under the Skorokhod $M_1$-topology, whereas our
proof employing the method of random time change readily shows a
stronger result stating that as $\epsilon$ tends to $0$,
$L^\epsilon$ converges \emph{almost surely} to $L$ under the
Skorokhod $M_1$-topology (see Theorem~\ref{th2} in Section
\ref{idea}). Here we remark that the proof of Theorem $3.1$
in~\cite{Ilya} is heavily dependent on the structure of the L\'{e}vy
process and cannot carry over to the case when $L$ is a general
semimartingale, in which case our proof however aptly applies.

The remainder of this paper is organized as follows. In Section~\ref{idea}, we use a special case to illustrate the key methodology used in our proof. In Section~\ref{WZ}, we prove Theorem~\ref{maintheorem}, the main result of this paper. For self-containedness, we recall in Appendix~\ref{skorokhod} some basic notions and results on the Skorokhod $J_1$ and $M_1$-topologies.

\section{A Special Case}\label{idea}

Note that if we set $b\equiv 0$, $f \equiv 1$ and $X_0=0$, then
$X^{\epsilon}$ is nothing but $L^{\epsilon}$. In this section, for
illustrative purposes, we will use the method of random time change
to establish the following theorem:
\begin{theorem}\label{th2}
$L^\epsilon$ converges almost surely to to the semimartingale $L$ under the Skorokhod
$M_1$-topology.
\end{theorem}

\begin{proof}
Let $[L]=[L,L]$ denote the quadratic variation of $L$, and let $[L]^c$ and $[L]^d$ denote its continuous and purely discontinuous parts, respectively. Define $ \gamma^0(t):=[L]^d(t)+t$, and for any $\epsilon > 0$, define
$$
\gamma^\epsilon(t):=\frac{1}{\epsilon}\int_{t-\epsilon}^t
([L]^d(s)+s)ds.
$$
It can be shown that for any $t \geq 0$ and any $\epsilon > 0$,
$\gamma^\epsilon(t)<\gamma^0(t)<\gamma^\epsilon(t+\epsilon)$. For any $\epsilon > 0$, let
$\varsigma^\epsilon$ be the generalized inverse of
$\gamma^\epsilon$, i.e., $\varsigma^\epsilon(t):=\inf\{s>0:
\gamma^\epsilon(s)>t\}$. It can also be shown that for any $t \geq 0$ and any $\epsilon >0$, $\varsigma^\epsilon(t)-\epsilon
<\varsigma^0(t) <\varsigma^\epsilon(t)$, which implies that
$\varsigma^\epsilon$ converges to $\varsigma^0$ uniformly over all
$t$ from any bounded interval.

For any $\epsilon >0$, define
$Z^\epsilon(t)=L^\epsilon_{\varsigma^\epsilon(t)}$; in other words, the new process $Z^\epsilon$ is the original process
$L^{\epsilon}$ reevaluated in the new time scale $\varsigma^\epsilon(\cdot)$. It can be easily verified that
$Z^\epsilon(t)$ is continuous in $t$.

The remainder of the proof consists of three steps as follows.

\textbf{Step 1:} In this step, we
will show that as $\epsilon$ tends to $0$, $\{Z^\epsilon(t)\}$
uniformly converges to a continuous process.

Defining
$$
\eta_-(t) = \sup\{s: \varsigma^0(s) <\varsigma^0(t) \}, \quad \eta^+(t) = \inf\{s: \varsigma^0(s) >\varsigma^0(t)\},
$$
letting $\{\tau_i, i\in \mathbb{N}\}$ denote the sequence of all the jump times of $L$, we will deal with the following two cases.

\textbf{Case 1:} $t\in [0,\gamma^0(\tau_1 -))$ or $t\in
(\gamma^0(\tau_i ),\gamma^0(\tau_{i+1} -))$ for some $i$.

In this case, we have $\eta_-(t)=\eta^+(t)$, and $L$ is continuous
at $\varsigma^0(t)$. Consequently,
\[
\lim_{\epsilon \rightarrow 0+}Z^\epsilon(t)=\lim_{\epsilon
\rightarrow
0+}L^\epsilon_{\varsigma^\epsilon(t)}=L_{\varsigma^0(t)}.
\]

\textbf{Case 2:} $t\in [\gamma^0(\tau_i -),\gamma^0(\tau_i)]$.

In this case, $\varsigma^0(t) \equiv \tau_i$ and  $\eta_-(t)\neq
\eta^+(t)$, and $L$ has a discontinuity at $\varsigma^0(t)$. It can
be shown that
\begin{equation}\label{z1}
\lim_{\epsilon \rightarrow 0+}Z^\epsilon(\gamma^0(\tau_i
-))=\lim_{\epsilon \rightarrow 0+}L^\epsilon_{\varsigma^\epsilon
\circ \gamma^0(\tau_i -) }=L_{\tau_i -}=L_{\varsigma^0(t)-},
\end{equation}
and moreover,
\begin{equation}\label{z2}
\lim_{\epsilon \rightarrow 0+}Z^\epsilon(\gamma^0(\tau_i
))=\lim_{\epsilon \rightarrow 0+}L^\epsilon_{\varsigma^\epsilon
\circ \gamma^0(\tau_i ) }=L_{\tau_i }=L_{\varsigma^0(t)}.
\end{equation}
And it follows from
\[
\frac{dZ^\epsilon(t)}{dt}=\frac{L_{\varsigma^\epsilon(t)}-L_{\varsigma^\epsilon(t)-\epsilon}}{[L]^d_{\varsigma^\epsilon(t)}-[L]^d_{\varsigma^\epsilon(t)-\epsilon}+\epsilon}
\]
that for any $t \in [\gamma^0(\tau_i -),\gamma^0(\tau_i)]$,
\begin{equation}\label{z3}
\lim_{\epsilon \rightarrow
0+}\frac{dZ^\epsilon(t)}{dt}=\frac{L_{\varsigma^0(t)}-L_{\varsigma^0(t)-}}{[L]^d_{\varsigma^0(t)}-[L]^d_{\varsigma^0(t)-}}
=\frac{L_{\varsigma^0(t)}-L_{\varsigma^0(t)-}}{\eta^+(t)-\eta_-(t)}.
\end{equation}
Consequently, it follows from \eqref{z1},\eqref{z2} and \eqref{z3}
that $\lim_{\epsilon \rightarrow 0+}Z^\epsilon(t) = Z(t)$ uniformly over all $t$ from any bounded interval, where $Z$ is continuous and admits
following expression
\begin{align*}
Z(t)=\begin{cases}
L_{\varsigma^0(t)},   &\text{if}\quad \eta_-(t)=\eta^+(t),\\
\frac{t-\eta_-(t)}{\eta^+(t)-\eta_-(t)}L_{\varsigma^0(t)}
+\frac{\eta^+(t)-t}{\eta^+(t)-\eta_-(t)} L_{\varsigma^0(t)-},
&\text{if}\quad \eta_-(t)\neq\eta^+(t).
\end{cases}
\end{align*}

\bigskip
\textbf{Step 2:} This step will lead to the conclusion that as $\epsilon$ tends to $0$, $\gamma^\epsilon$ converges almost surely to $\gamma^0$ under the Skorokhod $M_1$-topology. The proof of this step is postponed to next section (see Lemma~\ref{L2.1}).

\textbf{Step 3:} In this step, we will show that as $\epsilon$ tends to $0$, $L^\epsilon$ converges almost surely to $L$ under the Skorokhod $M_1$-topology, thereby completing the proof.

It follows from the facts that $\varsigma^0 \circ \gamma^0(t)=t$ and
$\gamma^0(t) \notin [\gamma^0(\tau_i-),\gamma^0(\tau_i))$ for any
$t, \tau_i$ that $Z_{\gamma^0(t)}\equiv L(t)$. Since $Z^\epsilon$
uniformly converges to $Z$, and $\gamma^\epsilon$ converges almost
surely to $\gamma^0$ under the Skorokhod $M_1$-topology, we conclude
that $L^\epsilon(\cdot)=Z^\epsilon_{\gamma^\epsilon(\cdot)}$
converges almost surely to $Z_{\gamma^0(\cdot)}=L(\cdot)$ under the
Skorokhod $M_1$-topology.

\end{proof}

\begin{remark}
Compared to Theorem~\ref{maintheorem}, Theorem~\ref{th2} deals with
a special setting yet yields a stronger result. On the other hand,
compared to Theorem $3.1$ in~\cite{Ilya}, as mentioned in
Section~\ref{S1}, Theorem~\ref{th2} treats a more general setting
and still produces a stronger result.
\end{remark}


\section{Proof of Theorem \ref{maintheorem}}\label{WZ}
%
The proof of Theorem $1.1$ roughly follows the framework laid out in the proof of Theorem~\ref{th2} and uses many notations defined therein.

For any $\epsilon > 0$, recall that $Z^{\epsilon}$ is defined as in
the proof of Theorem~\ref{th2}, and define $Y^{\epsilon}$ as
$Y^\epsilon(t)=X^\epsilon_{\varsigma^\epsilon(t)}$ for any $t$. It
can be easily verified that $\{Z^\epsilon(t)\}$ and
$\{Y^\epsilon(t)\}$ are continuous, and moreover $\{Y^\epsilon(t)\}$
is the unique solution to the following equation:
\begin{equation}\label{change0}
Y^\epsilon(t)=X_0+ \int_0^t b(Y^\epsilon(s)) d\varsigma^\epsilon(s)+\int_0^tf(Y^\epsilon(s))dZ^\epsilon(s), \quad 0 \leq t < \infty.
\end{equation}

We will first prove the following lemma.
\begin{lemma} \label{L1}
$Y^\epsilon$ converges in probability to a process $Y$ under the compact uniform topology.  Moreover, $Y$ is continuous and
satisfies
\begin{align*}
Y(t)&=X_0+\sum_i \left( \phi\left(f\Delta L(\tau_i), Y_{\gamma^0(\tau_i-)}, \frac{t\wedge\gamma^0(\tau_i)-\gamma^0(\tau_i-)}{|\Delta
L(\tau_i)|^2}\right)-Y_{\gamma^0(\tau_i -)}-f(Y_{\gamma^0(\tau_i-)})\Delta
L(\tau_i) \right)\\
& \quad \times I_{[\gamma^0(\tau_i -),+\infty)}+\int_0^t b(Y(s))
d\varsigma^0(s)+\int_0^tf(Y(s))dL_{\varsigma^0(s)}+\frac{1}{2}\int_0^tff'(Y(s))d[L]^c_{\varsigma^0(s)}.
\end{align*}
\end{lemma}

\begin{proof} The proof largely follows from that of Theorem $6.5$ in~\cite{Kurtz}, so we only give a sketch emphasizing the key steps.

As shown in Section~\ref{idea}, $\varsigma^\epsilon$ and $Z^\epsilon$ converge to $\varsigma^0$ and $Z$, respectively, both uniformly over any bounded interval, which immediately implies that $U^\epsilon \rightarrow U$ under the Skorokhod $J_1$-topology, where the processes $U^\epsilon$ and $U$ are defined as
$$
U^\epsilon(t) :=Z^\epsilon(t)-L_{\varsigma^\epsilon(t)}, \quad U(t) := Z(t)-L_{\varsigma^0(t)}.
$$
And we note that \eqref{change0} can be rewritten as
\begin{align}\label{change2}
Y^\epsilon(t) &=X_0+ \int_0^t b(Y^\epsilon(s))
d\varsigma^\epsilon(s)+\int_0^tf(Y^\epsilon(s))dL_{\varsigma^\epsilon(s)}+\int_0^tf(Y^\epsilon(s))dU^\epsilon(s) \nonumber\\
&=X_0+ \int_0^t b(Y^\epsilon(s))
d\varsigma^\epsilon(s)+\int_0^tf(Y^\epsilon(s))dL_{\varsigma^\epsilon(s)}
+f(Y^\epsilon(t))U^\epsilon(t)\nonumber\\
&\quad-\int_0^tf'(Y^\epsilon(s))U^\epsilon(s)dY^\epsilon(s)-[f(Y^\epsilon),U^\epsilon](t) \nonumber\\
&=X_0+ \int_0^t b(Y^\epsilon(s))
d\varsigma^\epsilon(s)+\int_0^tf(Y^\epsilon(s))dL_{\varsigma^\epsilon(s)}+f(Y^\epsilon(t))U^\epsilon(t) \nonumber\\
&\quad-\int_0^tf'(Y^\epsilon(s))f(Y^\epsilon(s))U^\epsilon(s)dZ^\epsilon(s)-\int_0^tf'(Y^\epsilon(s))
 b(Y^\epsilon(s))U^\epsilon(s) d{\varsigma^\epsilon}(s),
\end{align}
where we have used the fact that $[f(Y^\epsilon),U^\epsilon] \equiv 0$.

By~\cite[Lemma 6.3]{Kurtz}, we infer that
$\{\int_0^{\cdot}U^\epsilon(s)dZ^\epsilon(s)\}$ and
$\{{\varsigma^\epsilon}(\cdot)\}$ are ``good'' (see Kurtz-Protter~\cite{Kp91,Kurtz96}), and moreover
\begin{equation}\label{varsigma}
    \varsigma^\epsilon \rightarrow \varsigma^0
\end{equation}
uniformly, and in probability
\begin{equation}\label{UZ}
\int_0^tU^\epsilon(s)dZ^\epsilon(s) \rightarrow \frac{U(t)
^2-[L]_{\varsigma^0(t)}}{2}=\frac{(Z(t)-L_{\varsigma^0(t)}
)^2-[L]_{\varsigma^0(t)}}{2}
\end{equation}
under the Skorokhod $J_1$-topology. Then, parallel to the proof of Lemma $6.4$ in~\cite{Kurtz}, we deduce that
$f(Y^\epsilon(t))U^\epsilon(t)$ converges in distribution to $R(t)$
under the Skorokhod $J_1$-topology, where
\begin{equation}\label{R}
R(t)=\sum_i I_{[\gamma^0(\tau_i -),\gamma^0(\tau_i))}(t) f \left(
\phi\left(f\Delta L(\tau_i), Y_{\gamma^0(\tau_i -)},
\frac{t-\gamma^0(\tau_i-)}{|\Delta L(\tau_i)|^2}\right) \right)U(t);
\end{equation}
here, $\{\tau_i, i\in \mathbb{N}\}$, as in the proof of Theorem~\ref{th2}, is the sequence of all the jump times of $L$. Moreover, by the definition of
$U^\epsilon$ and $\varsigma^\epsilon$, we deduce that
\begin{equation}\label{f-vanish}
\int_0^tf'(Y^\epsilon(s)) b(Y^\epsilon(s))U^\epsilon(s)
d\varsigma^\epsilon(s) \rightarrow 0,
\end{equation}
as $\epsilon$ tends to $0$.

Now, combining \eqref{change2}-\eqref{f-vanish} as above, we
deduce from~\cite{Kurtz96} and \cite[Theorem 5.4]{Kp91} that
$Y^\epsilon$ converges in distribution to $Y$ under the Skorokhod
$J_1$-topology, where
$$
\hspace{-2mm} Y(t)=X_0+ \int_0^t b(Y(s))d\varsigma^0(s)+\int_0^tf(Y(s))dL_{\varsigma^0(s)}
+R(t)-\frac{1}{2}\int_0^tf'(Y(s))f(Y(s))d(U(s)^2-[L]_{\varsigma^0(s)}).
$$
Note that $U(t)$ can be further computed as
\begin{align*}
U(t) =Z(t)-L_{\varsigma^0(t)}=
\begin{cases}
0, &\text{if}\quad \eta_-(t)=\eta^+(t),\\
\frac{\eta^+(t)-t}{\eta^+(t)-\eta_-(t)}(L_{\varsigma^0(t)-}-L_{\varsigma^0(t)}),
&\text{if}\quad \eta_-(t)\neq\eta^+(t).
\end{cases}
\end{align*}
It then follows from the fact for any $t \in[\gamma^0(\tau_i-),
\gamma^0(\tau_i))$,$$ \varsigma^0(t)\equiv \tau_i, \quad
\eta^+(t)-\eta_-(t)= \Delta [L]^d_{\tau_i} =|\Delta L(\tau_i)|^2
$$
that
\begin{align*}
    Y(t)&=Y_{\gamma^0(\tau_i-)}+ \int_{\gamma^0(\tau_i-)}^t
b(Y(s))d\varsigma^0(s)+\int_{\gamma^0(\tau_i-)}^t f(Y(s))dL_{\varsigma^0(s)}\\
& \quad +f \left( \phi\left(f\Delta L(\tau_i), Y_{\gamma^0(\tau_i -)},
\frac{t-\gamma^0(\tau_i-)}{|\Delta L(\tau_i)|^2}\right)
\right)U(t)-\int_{\gamma^0(\tau_i-)}^t f'(Y(s))f(Y(s))
\frac{s-\gamma^0(\tau_i)}{\eta^+(s)-\eta_-(s)}ds\\
&=Y_{\gamma^0(\tau_i-)}+f(Y_{\gamma^0(\tau_i-)})\Delta L(\tau_i)+f
\left(\phi\left(f\Delta L(\tau_i), Y_{\gamma^0(\tau_i -)},
\frac{t-\gamma^0(\tau_i-)}{|\Delta L(\tau_i)|^2}\right)
\right)U(t)\\
& \quad -\int_{\gamma^0(\tau_i-)}^t f'(Y(s))f(Y(s))
\frac{s-\gamma^0(\tau_i)}{\eta^+(s)-\eta_-(s)}ds.
\end{align*}
Consequently,
\begin{align}\label{eqy}
Y(t)&=X_0+\sum_i \left( \phi\left(f\Delta L(\tau_i), Y_{\gamma^0(\tau_i
-)}, \frac{t\wedge\gamma^0(\tau_i)-\gamma^0(\tau_i-)}{|\Delta
L(\tau_i)|^2}\right)-Y_{\gamma^0(\tau_i -)}-f(Y_{\gamma^0(\tau_i
-)})\Delta
L(\tau_i) \right)\nonumber\\
&\hspace{-1cm}\quad \times I_{[\gamma^0(\tau_i
-),+\infty)}(t)+\int_0^t b(Y(s))
d\varsigma^0(s)+\int_0^tf(Y(s))dL_{\varsigma^0(s)}+\frac{1}{2}\int_0^tff'(Y(s))d[L]^c_{\varsigma^0(s)}.
\end{align}
Since $Y^\epsilon$, $Y$ are continuous, we infer that $Y^\epsilon$ converges in distribution to $Y$ under the compact uniform topology. Finally, using a similar argument in~\cite[Theorem 6.5]{Kurtz}, we conclude that $Y^\epsilon$ converges in probability to $Y$ under the compact uniform topology, and thereby completing the proof.
\end{proof}

\begin{remark}
With the added assumption that $b'$ is bounded and Lipschitz, the
proof of Theorem $6.5$ in~\cite{Kurtz} can be slightly modified to
prove that for almost all $t$, $X^\epsilon(t)$ converges in
probability to $X(t)$. By comparison, Lemma~\ref{L1} reaches the
same conclusion without the added assumption as above.
\end{remark}

The following lemma characterizes the convergence behavior of
$\gamma^\epsilon$.
\begin{lemma}\label{L2.1}
As $\epsilon$ tends to $0$, $\gamma^\epsilon$ converges almost surely to $\gamma^0$ under the Skorokhod
$M_1$-topology.
\end{lemma}
\begin{proof}
We first prove that $\gamma^\epsilon$ converges in probability to
$\gamma^0$ under the Skorokhod $M_1$-topology. It suffices to verify
the following convergence in probability
\[
V^\epsilon(t):= \frac{1}{\epsilon}\int_{t-\epsilon}^t [L]^d(s)ds
\rightarrow [L]^d(t)
\]
under the Skorokhod $M_1$-topology. To this end, by~\cite[Theorem
22, Page 66]{Protter}, the quadratic variation process $[L]^d$ of
the semimartingale $L^d$ is a c\`{a}dl\`{a}g, increasing and adapted
process. It then follows that the mapping $t\mapsto
\frac{1}{\epsilon}\int_{t-\epsilon}^t [L]^d(s)ds$ is monotone.
Obviously, by the definition of $w'$ (see \eqref{w'})$,
w'(V^\epsilon,\delta)=0$, which implies that for any fixed
$\Delta>0$, $\lim_{\delta \rightarrow 0+} \limsup_\epsilon
\PX(w'(V^\epsilon,\delta)> \Delta )=0$. Then, one verifies that for
any $t$, $V^\epsilon(t)$ converges in probability to $[L]^d(t)$.
So, by Proposition~\ref{p3}, we have shown that $V^\epsilon$ converges in probability $[L]^d$ under the Skorokhod $M_1$-topology.

Now we turn to prove that $\gamma^\epsilon$ converges almost surely to $\gamma^0$
under the Skorokhod $M_1$-topology. Using the fact that $[L]^d(t)$ is monotone in $t$ and the definition of $V^\epsilon$, we have that for any $0=\epsilon_\infty <\epsilon_2 <\epsilon_1$,
$$
[L]^d(t)=V^0(t)=V^{\epsilon_\infty}(t)\geq V^{\epsilon_2}(t) \geq V^{\epsilon_1}(t).
$$
It then follows from the definition of $d_{M_1, T}$ that
\[
d_{M_1, T}(V^{\epsilon_2},[L]^d) \leq d_{M_1, T}(V^{\epsilon_1},[L]^d),
\]
that is to say, for any fixed time $T$, almost all $\omega$ in
$\Omega$, $d_{M_1,T}(V^{\epsilon},[L]^d) $ is monotonically
increasing in $\epsilon$. Now, applying the proven fact $V^\epsilon$
converges in probability to the semimartingale $[L]^d$ under the
Skorokhod $M_1$-topology and~\cite[Lemma 2.5.4]{Ash}, we conclude
that $V^\epsilon$ converges almost surely to the semimartingale
$[L]^d$ under the Skorokhod $M_1$-topology. Consequently, $\gamma^\epsilon$ converges almost surely to $\gamma^0$ under the
Skorokhod $M_1$-topology, completing the proof.
\end{proof}

Henceforth, letting $\hat{X}(t)=Y_{\gamma^0(t)}$, we prove the following two lemmas.

\begin{lemma}\label{L3}
$X^\epsilon$ converges in probability to $\hat{X}$ under the Skorokhod $M_1$-topology.
\end{lemma}

\begin{proof}
The lemma immediately follows from Lemma~\ref{L1}, Lemma~\ref{L2.1}
and~\cite[Theorem 13.2.3]{W2002}.
\end{proof}

\begin{lemma}\label{L4}
$\hat{X}$ is the unique solution to the equation \eqref{mainsde}, and therefore $\hat{X}\equiv X$.
\end{lemma}

\begin{proof}
Since $\hat{X}(t)=Y_{\gamma^0(t)}$, by the equation \eqref{eqy}, we have
\begin{align*}
\hat{X}(t)&=X_0+\sum_i \left( \phi\left(f\Delta L(\tau_i),
Y_{\gamma^0(\tau_i
-)},\frac{\gamma^0(t)\wedge\gamma^0(\tau_i)-\gamma^0(\tau_i-)}{|\Delta
L(\tau_i)|^2}\right)-Y_{\gamma^0(\tau_i -)}-f(Y_{\gamma^0(\tau_i
-)})
\Delta L(\tau_i)\right)\\
&\quad \times I_{[\gamma^0(\tau_i
-),+\infty)}(\gamma^0(t))+\int_0^{\gamma^0(t)} b(Y(s))
d\varsigma^0(s)+\int_0^{\gamma^0(t)}f(Y(s))dL_{\varsigma^0(s)}\\
&\quad +\frac{1}{2}\int_0^{\gamma^0(t)}ff'(Y(s))d[L]^c_{\varsigma^0(s)}\\
&=X_0+\sum_i \left( \phi(f\Delta L(\tau_i), Y_{\gamma^0(\tau_i -)},
1)-Y_{\gamma^0(\tau_i -)}-f(Y_{\gamma^0(\tau_i -)})\Delta L(\tau_i)
\right) I_{[\gamma^0(\tau_i -),+\infty)}(\gamma^0(t))\\
&\quad +\int_0^t b(\hat{X}(s))
ds+\int_0^tf(\hat{X}(s-))dL(s)+\frac{1}{2}\int_0^{t}ff'(\hat{X}(s))d[L]^c(s)\\
&=X_0+\sum_{0< s \leq t} \left( \phi(f \Delta L(s), \hat{X}(s-), 1)-
\hat{X}(s-)-f
(\hat{X}(s-))\Delta L(s) \right)+\int_0^tb(\hat{X}(s))ds\\
& \quad
+\int_0^tf(\hat{X}(s-))dL(s)+\frac{1}{2}\int_0^{t}ff'(\hat{X}(s))d[L]^c(s)\\
& =X_0+ \int_0^t b(\hat{X}(s))ds + \int_0^tf(\hat{X}(s-))\diamond  d
L(s),
\end{align*}
where in the last equality, we have used the alternative definition
of a Marcus canonical equation in \eqref{m1}. So, $\hat{X}$ is the
solution to the equation \eqref{mainsde}, which, together with the
uniqueness of the solution to the equation \eqref{mainsde}, implies
that $\hat{X}\equiv X$.

\end{proof}

With all the lemmas as above, we are finally ready to prove Theorem~\ref{maintheorem}.

\textbf{Proof of Theorem~\ref{maintheorem}.} It follows from Lemma~\ref{L1} that $Y^\epsilon$ converges in probability to
$Y$ under the compact uniform topology. Moreover, it follows from Lemma~\ref{L3} that $X^\epsilon$ converges in
probability to $\hat{X}$ under the Skorokhod $M_1$-topology. The theorem then follows from Lemma~\ref{L4}, which asserts $\hat{X}\equiv X$. \qed


\section*{Appendix} \appendix

\renewcommand{\theequation}{\thesection.\arabic{equation}}

\section{Skorokhod topologies} \label{skorokhod}

Throughout this section, we fix $T > 0$.

The following $J_1$-metric has been defined by
Skorokhod~\cite{Skorokhod56}:
\[
d_{J_1, T}(x,y)=\inf_{\lambda\in \Lambda}\left\{ \sup_{0\leq t \leq
T}|x(t)-y((\lambda(t))|+\sup_{ s,t\in[0,T],s\neq t} \left|\log
\frac{\lambda(s)-\lambda(t)}{s-t}\right| \right\}, \quad x, y \in D([0, T], \R),
\]
where $\Lambda$ is the set of all the strictly increasing continuous functions mapping $[0,T]$ onto itself. The topology on $D([0, T], \R)$ induced by the $J_1$-metric is called the Skorokhod $J_1$-topology.

Skorokhod~\cite{Skorokhod56} also defined the $M_1$-metric using the
notion of \emph{completed graph} of a function. More precisely, for any $x\in D([0,T],\R)$, the completed graph of $x$, denoted by $\Gamma_x$, is
defined as
\[
\Gamma_x :=\{(t,z)\in [0,T]\times \R: z\in[[x(t-),x(t) ]]\},
\]
where $x(0-)$ is interpreted as $x(0)$, $[[z_1,z_2 ]]$ is the line segment connecting $z_1$ and $z_2$, i.e.,
\[
[[z_1,z_2 ]]=\{z\in \R: z=a z_1+(1-a)z_2 \text{ for some }
a \in[0,1] \}.
\]
Note that $\Gamma_x$ can be parametrically represented by the following continuous function
\[
(r,u): [0,1] \rightarrow \Gamma_z, \quad (r,u)(0)=(0,z(0)),
(r,u)(1)=(T,z(T)),
\]
which is nondecreasing with respect to the following order on $\Gamma_x$:
\[
(t_1,z_1)\leq (t_2,z_2) \Leftrightarrow  t_1 < t_2 \quad \text{or}
\quad (t_1=t_2 \text{ and } |x(t_1-)-z_1| \leq |x(t_2 -)-z_2|).
\]
Skorokhod~\cite{Skorokhod56} defined the $M_1$-metric as follows:
\[
\hspace{-2mm} d_{M_1,T}(x,y)=\inf_{(r_1,u_1)\in \Pi(x),(r_2,u_2)\in \prod(y)}\{
\sup_{0\leq t \leq 1}|r_1(t)-r_2(t)|+\sup_{0\leq t \leq 1}
|u_1(t)-u_2(t)| \}, \quad x, y \in D([0, T], \R),
\]
where $\Pi(\cdot)$ denotes the set of all parametric representations of an element in $D([0, T], \R)$. The topology on $D([0, T], \R)$ induced by the $M_1$-metric is called the Skorokhod $M_1$-topology.

Noting that the limit of a sequence of continuous functions under either the uniform or the Skorokhod $J_1$-topology is continuous, we remark that, when approximating a c\`{a}dl\`{a}g function using continuous functions, the Skorokhod $M_1$-topology can be particularly useful. For example, for any $n \geq 1$, let
$$
x(t)=I_{[1/2, 1]}(t), \quad x^n(t)=n(t-1/2+1/n)I_{[1/2-1/n, 1/2]}(t)+I_{[1/2, 1]}(t), \quad 0 \leq t \leq 1.
$$
One can verify that, as $n$ tends to infinity, $x^n(t) \rightarrow x(t)$ in $D([0,1],\R)$ under Skorokhod $M_1$-topology but not under the Skorokhod $J_1$-topology.

The following theorem is well known; see, e.g., ~\cite[Theorem 3.2]{PR15}.
\begin{theorem} \label{p3}
Let $\{X^n : n \geq 1\}$ be a sequence of $D([0,T],\R)$-valued random elements. $\{X^n\}$ converges in probability to $X$ under the Skorokhod $M_1$-topology if and only if
\begin{enumerate}
\item[1)] for any $t \in [0, T]$, $X^n(t)$ converges in probability to $X(t)$;
\item[2)] and for any fixed $\Delta>0$,
$$
\lim_{\delta \rightarrow 0+} \limsup_n \PX(w'(X^n,\delta) > \Delta )=0,
$$
where
\begin{equation} \label{w'}
w'(x,\delta):= \sup_{0\leq t_1 <t< t_2 \leq 1; t_2-t_1< \delta}
\inf_{a\in[0,1]}|x(t)-(ax(t_1)+(1-a)x(t_2))|.
\end{equation}
\end{enumerate}
\end{theorem}


\end{document}